\documentclass[reqno,centertags,11pt]{amsart}
\usepackage{verbatim}
\usepackage{tikz, enumerate, array, amssymb, amsmath}
\usepackage{graphicx, epsfig }
\usepackage{amsthm}
\usepackage{amsbsy}
\usepackage{amsfonts}

\usepackage{mathtools}

\newcommand{\Hmm}[1]{\leavevmode{\marginpar{\tiny%
$\hbox to 0mm{\hspace*{-0.5mm}$\leftarrow$\hss}%
\vcenter{\vrule depth 0.1mm height 0.1mm width \the\marginparwidth}%
\hbox to 0mm{\hss$\rightarrow$\hspace*{-0.5mm}}$\\\relax\raggedright #1}}}

\newcommand{\N}{{\mathbb{N}}}

\newcommand{\R}{{\mathbb{R}}}

\newcommand{\C}{{\mathbb{C}}}

\newcommand{\Z}{{\mathbb{Z}}}

\newcommand{\f}{\frac}

\newcommand{\ol}{\overline}

\newcommand{\loc}{\text{\rm{loc}}}

\newcommand{\hatt}{\widehat}
\newcommand{\beq}{\begin{equation}}
\newcommand{\eeq}{\end{equation}}
\newcommand{\bdm}{\begin{displaymath}}
\newcommand{\edm}{\end{displaymath}}
\newcommand{\ba}{\begin{align}}
\newcommand{\ea}{\end{align}}
\newcommand{\bpf}{\begin{proof}}
\newcommand{\epf}{\end{proof}}

\newcommand{\la}{\langle}
\newcommand{\ra}{\rangle}

\newcommand{\veps}{\varepsilon}

\newcommand{\re}{\mathrm{Re}}
\newcommand{\im}{\mathrm{Im}}

\newcommand{\dav}{{d_{\mathrm{av}}}}

\allowdisplaybreaks

\newcommand{\calC}{\mathcal{C}}

\newcommand{\calI}{\mathcal{I}}

\newcommand{\calS}{{\mathcal{S}}}

\newtheorem{theorem}{Theorem}
\newtheorem{proposition}[theorem]{Proposition}
\newtheorem{lemma}[theorem]{Lemma}
\newtheorem{corollary}[theorem]{Corollary}

\theoremstyle{definition}

\newtheorem{remark}[theorem]{Remark}

\newcounter{theoremi}[theorem]

\numberwithin{theorem}{section}
\numberwithin{equation}{section}

\newcounter{assumptions}

\newcounter{smalllist}

\newcounter{listi}
\newenvironment{theoremlist}{\begin{list}{{\rm(\roman{listi})}}{%
\setlength{\topsep}{0mm}\setlength{\parsep}{0mm}\setlength{\itemsep}{0mm}%
\setlength{\labelwidth}{1.5em}\setlength{\leftmargin}{1.7em}\usecounter{listi}%
}}{\end{list}}

\newcounter{smallenum}

\begin{document}

\title[ DMNLS with lumped amplification ]{On Dispersion managed nonlinear Schr\"odinger equations with lumped amplification }
\author[M.--R. Choi, Y. Kang, Y.--R. Lee]{Mi-Ran Choi$^\dag$, Younghoon Kang$^\ddag$, Young-Ran Lee$^\ddag$}
\address{$^\dag$ Research Institute for Basic Science, Sogang University, 35 Baekbeom--ro (Sinsu--dong),
    Mapo-gu, Seoul 04107, South Korea.}%
\email{rani9030@sogang.ac.kr}


\address{$^\ddag$ Department of Mathematics, Sogang University, 35 Baekbeom--ro (Sinsu--dong),
    Mapo--gu, Seoul 04107, South Korea.}%
\email{kkyh0409@sogang.ac.kr, younglee@sogang.ac.kr}

\thanks{\copyright 2021 by the authors. Faithful reproduction of this article,
       in its entirety, by any means is permitted for non-commercial purposes}

\begin{abstract}
We show the global well--posedness of the nonlinear Schr\"odinger equation with periodically varying coefficients and a small parameter $\veps>0$, which is used in optical--fiber communications. We also prove that the solutions converge to the solution for the Gabitov--Turitsyn or averaged equation as $\veps$ tends to zero.
\end{abstract}

\maketitle

\section{Introduction }\label{introduction}

We consider the nonlinear Schr\"odinger equation (NLS) with periodically
varying coefficients
\beq \label{eq:intro}
i\partial_t u +d(t) \partial_x^2 u +c(t)|u|^2 u = 0
\eeq
which describes the behavior of a signal transmitted on an optical--fiber cable.
Here, $x$  denotes the (retarded) time, $t$ the position along the cable, and periodic functions $d(\cdot)$ and $c(\cdot)$ the dispersion and the fiber loss/amplification along the cable respectively.

The original evolution of optical pulses in a dispersion managed system with lumped amplification is described by the nonlinear Schr\"odinger equation
\beq\label{eq:original}
i\partial _t E+d(t) \partial_x^2 E+ |E |^2 E = i g(t)E. \notag
\eeq
The fiber loss and  amplification coefficient
along the cable is given by
\[
g(t)= -\f{\Gamma}{2}+\Gamma \sum_{j=1}^\infty \delta (t-t_j),
\]
where $\Gamma >0$ is the fiber loss, $t_j$ corresponds to the location of amplifiers, and $\delta(\cdot)$ is the Dirac delta function. For more information on this equation, see, e.g., \cite{Ag}. Taking
\[
E(x,t)= u(x,t)\exp\left(\int _0 ^ t g(t')dt'\right),
\]
we obtain equation \eqref{eq:intro}
with
$c(t)= \exp \left(2\int_0 ^t g(t') dt'\right)$
provided that $g$ is a periodic function with mean zero.

The dispersion management with alternating sections of positive and negative dispersion in fibers was introduced in 1980, see  \cite{LKC80}. It was successful to transfer the data at ultra--high speed over long distances, see, e.g., \cite{AB98,CGTD93,GT96a,GT96b,KH97,T03}.
For more information on the dispersion management, see \cite{TBF} and references therein.


In the strong dispersion management regime, the dispersion is given by
\beq\label{eq:d(t)}
d(t)= \dav +\f{1}{\veps}d_0 \left(\f{t}{\veps}\right),  \notag
\eeq
where $d_0(\cdot)$ is the mean zero part of the dispersion which is a $2-$periodic function satisfying
 $$
 d_0(t) = \left\{
            \begin{array}{ll}
              \phantom {-}1, & 0\leq t < 1, \\
              -1, & 1\leq t <2,
            \end{array}
          \right.
 $$
$\dav \in \R$ the average dispersion over one period, and $\veps>0$ a small parameter. The fiber loss and amplification is defined to be
\[
c(t)= G \left(\f{t}{\veps}\right),
\]
where $G$ is also a $2-$periodic function given by
\beq\label{eq:G(t)}
G(t)= \exp \Bigl(2 \int_{0}^{t} \Bigl(-\f{\Gamma}{2}+\Gamma \sum_{j \in \Z}\delta (t'-2j)\Bigr)\,dt' \Bigr). \notag
\eeq

The first main result of this paper is the well--posedness of the Cauchy problem
\beq \label{eq:intro_main}
\begin{cases}
   i\partial_t u +\Bigl(\dav +\f{1}{\veps}d_0 \left(\f{t}{\veps}\right)\Bigr) \partial_x^2 u +
   G\left(\f{t}{\veps}\right)|u|^2 u = 0,\\
   u(x,0)=u_0(x).
\end{cases}
\eeq

\begin{theorem}[Global well--posedness]\label{thm:globalwellposedness}
Let $\dav\in \R$. For every $u_0 \in H^1(\R)$, there exists a unique solution $u\in \calC(\R, H^1(\R))$ of \eqref{eq:intro_main}. Moreover, $u$ depends continuously on the initial datum in the following sense. For every $M>0$, the map $u_0 \mapsto u(t)$ from $H^1(\R)$ to $\calC([-M,M], H^1(\R))$  is locally Lipschitz continuous.
\end{theorem}

Now we change the variables $u=T_{D(t/\veps)}v$ in \eqref{eq:intro_main} to obtain
\beq\label{eq:intro transformed equation}
\begin{cases}
   i\partial_t v + \dav \partial_x^2 v+ G\left(\f{t}{\veps}\right)T_{D(t/\veps)}^{-1}\Bigl(|T_{D(t/\veps)}v|^2T_{D(t/\veps)}v\Bigr)=0,\\
   v(x,0)=u_0(x),
\end{cases}
\eeq
where $D(t)=\int_0^t d_0(t')dt'$ and $T_t$ is the solution operator for the free Schr\"{o}dinger equation in dimension one.
Note that
since $d_0$ is a $2-$periodic function with mean zero, $D$ is also $2-$periodic and, therefore, the map $t \mapsto T_{D(t/\veps)}$ is $2\veps-$periodic.
\\
For small $\veps>0$, that is, in the regime of strong dispersion management, equation \eqref{eq:intro transformed equation} contains the fast oscillating terms $T_{D(t/\veps)}$ and $G(t/\veps) $ in the nonlinearity and hence Gabitov and Turitsyn suggested averaging the equation over one period, see \cite{GT96a, GT96b}. This yields
the following ``averaged" equation
\begin{equation}
\begin{aligned}
i\partial_t v + \dav \partial_{x}^2 v+ \f{1}{2}\int_0^2 G(\tau) T_{D(\tau)}^{-1} \Bigl(|T_{D(\tau)} v|^2 T_{D(\tau)} v\Bigr)d\tau=0. \notag
\end{aligned}
\end{equation}
We make the change of variables $D(\tau)=r$, then we have
\beq\label{eq:intro averaged}
i\partial_t v+ \dav \partial_{x}^2 v +  \int_0^1    T_{r}^{-1} \Bigl(|T_r v|^2 T_r v\Bigr) \psi(r)dr=0,
\eeq
where
$$
\psi(r)=e^{-\Gamma}\cosh \Gamma(r-1).
$$
The well--posedness of the averaged equation \eqref{eq:intro averaged} in $H^s(\R)$ for all $s \ge 0$ is proved in \cite{AK} for a general dispersion profile. For more general nonlinearities including saturated nonlinearities, see  \cite{CHLW}.

\medskip
The averaging procedure is rigorously justified in \cite{ZGJT01} when the fiber loss and amplification are not present. More precisely, it is shown that for $\veps>0$, the solutions of \eqref{eq:intro transformed equation} and \eqref{eq:intro averaged} with the same initial datum in $H^s(\R)$ stay $\veps$--close in $H^{s-3}(\R)$ for a long time in $O(\veps^{-1})$  when $s$ is sufficiently large. Note that the convergence is not shown in $H^s(\R)$ where the solutions exist. However, we prove that the solutions for \eqref{eq:intro transformed equation} converge to the solution for \eqref{eq:intro averaged} in $H^1(\R)$, where the solutions exist, as $\veps \to 0$.

{
\begin{theorem}[Averaging Theorem]\label{thm:intro averaging}
Let $\dav\in \R$, $M>0$ and  $ v \in \calC(\R, H^1(\R))$ be the solution of the averaged equation \eqref{eq:intro averaged} with the initial datum $ v_0\in H^1(\R)$.
Then there exist $C>0$ and $\veps_0>0$ such that if $0< \veps \leq \veps_0$ and $\|u_0-v_0\|_{H^1(\R)} \leq \veps$,
then
\beq \label{ineq:averaging thm}
\|v_\veps - v\|_{\calC([-M, M],H^{1}(\R))}\leq C \veps,
\eeq
where $v_\veps$ is the solution of \eqref{eq:intro transformed equation} with the initial datum $u_0\in H^1(\R)$
\end{theorem}

In our main theorems, Theorems \ref{thm:globalwellposedness} and \ref{thm:intro averaging}, we prove the well--posedness of the Cauchy problem \eqref{eq:intro_main} and the validity of the averaging process in the strong dispersion, $\f 1\veps d_0(\f t\veps)$, while the well--posedness of the averaged equation is already proved in \cite{AK}. There are analogous results for the fast dispersion,  $d_0(\f t\veps)$, and the random dispersion, $\f 1\veps d_0(\f t{\veps^2})$, with some centered stationary random process $d_0$, in \cite{ASS} and \cite{BD}, respectively.}

 We here remark some provable facts which are not dealt with in this paper.
Related to standing wave solutions $v(x,t)=e^{i\omega t}f(x)$, $\omega \in \R$, of the averaged equation \eqref{eq:intro averaged}, a constrained minimization problem is well studied. The existence of minimizers can be found in \cite{CHLT,HL} when $\dav\geq 0$. One can easily show that every minimizer is a weak solution of the corresponding Euler--Lagrange equation. Each weak solution and its Fourier transform decay exponentially, which can be proven by modifying the proofs in \cite{EHL,GH} a little. Particularly, every minimizer is smooth.
Moreover, the set of ground states is orbitally stable, see \cite{CHLW,HKS} as well as \cite{CL}. In the case $\dav>0$, using the averaging theorem and the orbital stability, it is possible to obtain the stable soliton--like solution for \eqref{eq:intro_main}.

\medskip

The paper is organized as follows.
In Section \ref{sec:well-posedness}, we prove Theorem \ref{thm:globalwellposedness}, the global well--posedness result in $H^1(\R)$. We start by showing the local well--posedness in $H^1(\R)$ and the global existence in $L^2(\R)$. Although we do not have the energy conservation, we prove the existence of a global solution in $H^1(\R)$ based on the mass conservation and the boundedness of the mixed norm of the solution for Strichartz admissible pairs. In Section \ref{sec:Averaging theorem}, we prove Theorem \ref{thm:intro averaging}, the averaging theorem.
In Appendix \ref{sec:appendix}, we gather basic properties of the free Schr\"odinger time evolution to prove Theorem \ref{thm:globalwellposedness}.

\section{Well--posedness}\label{sec:well-posedness}

To begin with, let us introduce some notations. For $1\leq p <\infty$, we use $L^p(\R)$ to denote the Banach space of functions $f$ whose norm
  \[
  \|f\|_{L^p} :=\left(\int _\R |f(x)|^p dx\right)^{\f1p}
  \]
is finite with the essential supremum instead when $p=\infty$.
  The space $L^2(\R)$ is a Hilbert space with scalar product given by $\la f,g\ra=\int_\R f(x) \overline{g(x)} dx$.
We use $L_t^q(J, L_x^p(I))$ to denote, for $1 \le p, q < \infty$ and intervals $I, J\subset \R$, the Banach space of all functions $u$ with the mixed norm
$$
\|u\|_{L_t^q(J, L_x^p(I))}:=\left(\int _J \left(\int _I |u(x,t)|^p dx\right)^{\f{q}{p}}dt\right)^{\f{1}{q}}\ .
$$
If $p=\infty$ or $q=\infty$, use the usual modification. For notational simplicity, we use $L^q(J,  L^p)$ for $L_t^q(J, L_x^p(\R))$. We say that $u \in L^q_{\loc}(J,  L^p)$ when $u \in L^q(\widetilde{J},  L^p)$ for every bounded interval $\widetilde{J} \subset J$.

The Fourier transform on $\R$ is defined by
\[
	\hatt{f}(\xi):=\f{1}{\sqrt{2\pi}}\int_\R e^{-ix\xi}f(x)\, dx
\]
for $f\in \calS(\R)$, the Schwartz space of infinitely smooth, rapidly decreasing functions.
For $s\in \R$, the Sobolev space $H^s(\R)$ is defined as the space of all tempered distributions $f\in S'(\R)$ for which
\[
\|f\|_{H^s}:=\left(\int_\R (1+\xi^2)^s |\hatt{f}(\xi)|^2 d\xi\right)^{1/2}< \infty.
\]

For a Banach space $X$ with norm $\|\cdot\|_X$ and an interval $J$, $C(J, X)$ is the space of all continuous functions $u:J \to X$. When $J$ is compact, it is a Banach space with norm
$$
\|u\|_{C(J, X)}=\sup _{t\in J}\|u(t)\|_{X}
$$
and $C^1(J, X)$ is the Banach space of all continuously differentiable functions $u:J \to X$.

Let $T_t$  denote the solution operator for the free Schr\"{o}dinger equation in spatial dimension one.
In terms of the Fourier transform, this is given by
\[
\widehat{ T_t f}(\xi)=\widehat{ e^{it\partial_x^2}f}(\xi)=e^{-it\xi^2} \hatt f(\xi)
\]
for $f\in \calS(\R)$,
 thus, one can express
\[
T_t f(x)=e^{it \partial_x^2}f(x)= \f{1}{\sqrt{2\pi}} \int _\R e^{ix\xi} e^{-i t\xi^2  }\widehat{f}(\xi)d\xi.
\]
 It is a unitary operator on $L^2(\R)$ and also on $H^1(\R)$. Therefore, for every $t\in \R$,
$$
\|T_t f\|_{L^2}=\|f\|_{L^2} \quad \mbox{and} \quad \|T_t f\|_{H^1}=\|f\|_{H^1}.
$$
We use the notation $f \lesssim g$ when there exists a positive constant $C$ such that $f \le Cg$.

\medskip
Now we prove the well--posedness of the Cauchy problem \eqref{eq:intro_main}.
Since the proof does not rely on the factor $\veps$ in \eqref{eq:intro_main}, we only consider the case $\veps=1$,
\beq\label{eq:mainCauchy}
\begin{cases}
   \displaystyle{i\partial_t u+ \bigl(\dav+d_0(t)\bigr) \partial_x^2u+ G(t)|u|^2u=0},\\
   u(x,0)=u_0(x),
\end{cases}
\eeq
or, equivalently,
\beq\label{eq:transformedCauchy}
\begin{cases}
   i\partial_t v + \dav \partial_x^2 v+ G(t) T_{D(t)}^{-1}\Bigl(|T_{D(t)}v|^2T_{D(t)}v\Bigr)=0,\\
   v(x,0)=u_0(x),
\end{cases}
\eeq
where $u=T_{D(t)}v$.
We first prove the local existence of a unique solution for the integral equation of \eqref{eq:transformedCauchy},
\beq \label{eq:duhamel fomula}
v(t)=e^{it \dav \partial_x^2} u_0 + i \int _0 ^t e^{i (t-t')\dav \partial_x^2} Q(v(t'))dt',
\eeq
with  $u_0 \in H^1(\R)$, where
$$
 Q(v)(t)= G(t)T_{D(t)}^{-1}\Bigl(|T_{D(t)}v(t)|^2T_{D(t)}v(t)\Bigr).
$$
Here and below, we use $C$ to denote various constants.

\begin{lemma}\label{lem:Q}
For every $f, g\in H^1(\R)$, we have
\beq\label{ineq:bounded}
\|Q(f)(t)\|_{H^1}\lesssim \|G\|_{L^\infty}\|f\|^3_{H^1}
\eeq
and
\beq \label{ineq:difference}
\|Q(f)(t)-Q(g)(t)\|_{H^1}\lesssim \|G\|_{L^\infty} (\|f\|^2_{H^1}+\|g\|^2_{H^1})\|f-g\|_{H^1}.\eeq
\end{lemma}
\begin{proof}
We define, for every $f_1, f_2, f_3\in H^1(\R)$,
\[
Q(f_1, f_2, f_3)(t):= G(t)T_{D(t)}^{-1}\Bigl(T_{D(t)}f_1 \overline{T_{D(t)}f_2}T_{D(t)}f_3\Bigr),
\]
which is multi--linear. Note that $Q(f)(t)=Q(f,f,f)(t)$ for any $f \in H^1(\R)$.
Since
\[
\|fg\|_{H^1}\leq C \|f\|_{H^1}\|g\|_{H^1}
\]
and $T_{D(s)}$ is unitary on $H^1(\R)$, we obtain
\beq\label{ineq:Q}
\|Q(f_1, f_2, f_3)(t)\|_{H^1}\lesssim \|G\|_{L^\infty}\|f_1\|_{H^1}\|f_2\|_{H^1}\|f_3\|_{H^1}
\eeq
which proves \eqref{ineq:bounded}.
Observing
\[
Q(f)(t)-Q(g)(t)=Q(f-g, f,f)(t)+Q(g, f-g, f)(t)+ Q(g, g, f-g)(t),
\]
one can easily obtain \eqref{ineq:difference} from \eqref{ineq:Q}.
\end{proof}

\begin{proposition}\label{prop:H^slocal theory}
Let $\dav\in \R$. For every $K>0$, there exist $M_\pm>0$ such that for every initial datum $u_0 \in H^1(\R)$ with $\|u_0\|_{H^1}\leq K$,  there is a unique solution $v\in \calC([-M_-, M_+], H^1)$ of \eqref{eq:duhamel fomula}.
Moreover,
\beq
\|v(t)\|_{ H^1}\leq 2 K  \quad \text{   for all  } t\in [-M_-, M_+]. \notag
\eeq

\end{proposition}

\begin{corollary} \label{cor:maximal existence in H^1}
Let $\dav\in \R$. For any initial datum $u_0\in H^1(\R)$, there exist maximal life times $T_\pm\in (0, \infty]$ such that there is a unique solution $v\in \calC((-T_-, T_+), H^1)$ of \eqref{eq:duhamel fomula}.
Moreover, the blow--up alternative for solutions holds:
 if $T_+$ is finite,
  	then
 	\begin{equation}
 		\lim_{t\to T_+} \|v(t)\|_{H^1} =\infty  \notag
 	\end{equation}
   and if $T_-$ is finite,
   then
 	\begin{equation}
 		\lim_{t\to -T_-} \|v(t)\|_{H^1} =\infty . \notag
 	\end{equation}
\end{corollary}

\begin{proof}
[Proof of Proposition \ref{prop:H^slocal theory}]
Without loss of generality, we assume that $t>0$.

To prove the existence of a solution, we use a fixed point argument.
For each $M>0$ and $a>0$, let
\[
B_{M,a}=\{v\in \calC([0, M], H^1)  \; : \; \|v\|_{\calC([0, M],H^1)}\leq a\}
\]
be equipped with the distance
$$
d(v,w)=\|v-w\|_{\calC([0, M],H^1)}.
$$
Let $K>0$ and $u_0\in H^1(\R)$ with $\|u_0\|_{H^1}\leq K$ be fixed. Define the map $\Phi$ on $B_{M,a}$ by
\bdm
\Phi(v)(t)=e^{it \dav \partial_x^2 }u_0+i\int _0 ^t e^{i(t-t') \dav \partial_x^2 } Q(v)(t')dt'.
\edm
It follows from Lemma \ref{lem:Q} that if $v(t), w(t) \in H^1(\R)$, then
\[
\begin{aligned}
\|\Phi(v)(t)\|_{H^1}
& \leq \|u_0\|_{H^1} + \int _0 ^t  \|Q(v)(t')\|_{H^1}dt'
\leq \|u_0\|_{H^1} + C\int _0 ^M \|v(t')\|^3_{H^1}dt'
\end{aligned}
\]
and
\beq
\begin{aligned}\label{ineq:Lips}
\| \Phi(v)(t)-\Phi(w)(t)\|_{H^1}&\leq \int_0 ^t \|Q(v)(t')-Q(w)(t')\|_{H^1} dt' \\
&\leq C \int_0 ^M \left(\|v(t')\|_{H^1}^2 +\|w(t')\|_{H^1}^2\right)  \|v(t')-w(t')\|_{H^1}dt'
\end{aligned}
\eeq
for every $0\leq t \leq M $.
Therefore, there is a positive constant $C$ such that for all $v, w\in B_{M,a}$,
\[
\|\Phi(v)\|_{\calC([0,M], H^1)}\leq K +CM a^3
\]
and
\[
d(\Phi(v),\Phi(w))\leq CMa^2 d(v,w).
\]
Now set $a=2K$ and choose $M_+>0$ satisfying
\[
 CM_+(2K)^2 <  \f{1}{2},
\]
then we obtain that $\Phi$ is a contraction from $B_{M_+,2K}$ into itself.
Thus,
Banach's contraction mapping theorem shows that there exists a
unique solution $v$ of \eqref{eq:duhamel fomula} in $B_{M_+,2K}$ and, moreover,
\[
\|v\|_{\calC([0, M_+],H^1)}\leq 2K.
\]

To show the uniqueness of a solution,
let $v_1, v_2 \in \calC([0, M_+],H^1)$ be solutions of \eqref{eq:duhamel fomula}.
Then it follows from \eqref{ineq:difference} that for every $t \in [0,M_+]$
\beq\label{ineq:uniqueness}
\|v_1(t)-v_2(t)\|_{H^1}
\leq C \left(\|v_1\|^2_{\calC([0, M_+],H^1)}+\|v_2\|^2_{\calC([0, M_+],H^1)}\right) \int_0 ^t \|v_1(t')-v_2(t')\|_{H^1}dt'
\eeq
which implies $\|v_1-v_2\|_{\calC([0,M_+],H^1)} =0$.
\end{proof}

\begin{proof}[Proof of Corollary \ref{cor:maximal existence in H^1}]
Given the initial datum $0 \neq u_0 \in H^1(\R)$, let us define the maximal life time $T_+$ by
\begin{align*}
&T_+=\sup\left\{ M>0: \text{a unique solution of}\ \eqref{eq:duhamel fomula}\ \text{exists in}\ \calC([0, M],H^1) \right\}.
\end{align*}
Then it immediately follows from Proposition \ref{prop:H^slocal theory} that $T_+ \in (0, \infty]$.

Note that if a solution exists in $\calC([0, M],H^1)$ for any $M>0$, then it is a unique solution in $\calC([0, M],H^1)$, by the same argument in the proof of the uniqueness in Proposition \ref{prop:H^slocal theory}. Thus, there exists a unique solution $v\in \calC([0,T_+),H^1)  $  of \eqref{eq:duhamel fomula}. To prove the blow--up alternative, let $T_+ < \infty$. Suppose to the contrary that there exist a positive number $K$ and a sequence $\{t_j\}$ in $(0, T_+)$ such that $\|v(t_j)\|_{H^1}\leq K$ and $t_j \to T_+$  as $j \to \infty$. Then, by Proposition \ref{prop:H^slocal theory}, there exists $M>0$ such that a unique solution of  \eqref{eq:duhamel fomula} with the initial datum $v(t_j)$ exists in $\calC([t_j, t_j+M], H^1)$ for all $j$. Since we can choose $j^*$ such that $t_{j^*} + M >T_+$, this contradicts the definition of $T_+$. The case of $T_-$ can be done similarly.
\end{proof}

Next, we show the continuous dependence of the solutions for \eqref{eq:duhamel fomula} on the initial data to finish the local well--posedness. Indeed, the map $u_0 \mapsto v(t)$ is locally Lipschitz continuous on $H^1(\R)$.

\begin{proposition}\label{prop:continuousdependence}
Let $\dav\in \R$. For every $K>0$, there exists a positive constant $C$ such that for all initial data $ v_0, w_0 \in H^1(\R)$ with $\|v_0\|_{H^1}, \|w_0\|_{H^1}\leq K$,  we have
$$
\|v-w\|_{\calC([-M_-,M_+], H^1)}\leq e^{C \max(M_-, M_+)}\|v_0-w_0\|_{H^1},
$$
where $v$ and $w$ are the corresponding local solutions of \eqref{eq:duhamel fomula} with initial data $ v_0$, $w_0 $
on the time interval $[-M_-, M_+]$ of existence, guaranteed by Proposition \ref{prop:H^slocal theory}.
\end{proposition}

\begin{proof}
We consider positive $t$ only and fix $t \in (0,M_+]$.
From Proposition \ref{prop:H^slocal theory}, we know that
$$
\|v(t)\|_{H^1}\le 2K \quad\mbox{and}\quad \|w(t)\|_{H^1}\leq 2K .
$$
By a similar argument of \eqref{ineq:Lips}, we obtain
\begin{align*}
\|v(t)-w(t)\|_{H^1}
& \leq \|v_0-w_0\|_{H^1}+\int_0 ^t \|Q(v)(t')-Q(w)(t')\|_{H^1} dt'\\
& \leq \|v_0-w_0\|_{H^1}+C\int_0 ^t \Bigl(\|v(t')\|_{H^1}^{2}+\|w(t')\|_{H^1}^{2}\Bigr) \| v(t')-w(t')\| _{H^1} dt'\\
&\leq \|v_0 -w_0\|_{H^1} +CK^2 \int _0 ^t \| v(t')-w(t')\| _{H^1}dt'.
\end{align*}
Thus, it follows from Gronwall's inequality that
$$
\|v(t)-w(t)\|_{H^1}\leq e^{CK^2t} \|v_0-w_0\|_{H^1}\leq e^{CK^2M_+} \|v_0-w_0\|_{H^1},
$$
which completes the proof.
\end{proof}

\begin{remark}\label{rem:blowup}
     If we define the energy $E(v(t))$ of the solution $v$ for \eqref{eq:transformedCauchy}  by
\[
E(v(t)) =\f{\dav}{2} \|\partial_x v(t)\|_{L^2}^2 - \f{G(t)}{4}\int _\R  |T_{D(t)}v(t)|^4 dx,
\]
then, however, the energy is neither conserved nor decreasing.
Indeed, its derivative is given by
\beq \label{energy derivative}
\f{dE(v(t)) }{dt} = - \f{1}{4}G'(t)\int _\R  |T_{D(t)}v(t)|^4 dx
\eeq
for all $t \in \R \setminus 2\Z $. Note that $E(v(t))$ is not differentiable nor continuous at $t \in 2\Z$.\\
If there is no fiber loss nor amplification, i.e., $G \equiv 1$, then $T_\pm=\infty$  by the conservation
of energy and the blow--up alternative, which immediately gives the global well--posedness. However, as you see in \eqref{energy derivative}, the energy is no longer conserved in our case.
\end{remark}

\medskip

Now we consider the integral form of the Cauchy problem \eqref{eq:mainCauchy}
\beq \label{eq:duhamel fomula for o}
u(t)=U(0,t)u_0 + i \int _{0} ^t U(t',t)\Bigl( G(t')|u(t')|^2u(t')\Bigr) dt'.
\eeq
Here and below, $U(0,t)$ is the solution operator for the linear Schr\"odinger equation associated with \eqref{eq:mainCauchy}, i.e., for every $f\in L^2(\R)$, $U(0,t)f$ solves the initial value problem
\[
\begin{cases}
   i\partial_t w + \bigl(\dav+d_0(t)\bigr)\partial_x^2 w =0,\\
   w(x,0)=f(x).
\end{cases}
\]
Next we define, for all $t_0, t \in \R$,
\beq\label{eq:linear-propagator}
U(t_0,t):= U(0,t)\bigl(U(0,t_0)\bigr)^{-1}
\eeq
on $L^2(\R)$.
Then $U(t_0,t)$ is unitary on $L^2(\R)$ and also on $H^1(\R)$, and therefore, for every $t, t_0\in \R$,
\beq \label{eq:unitary}
\|U(t_0,t)f\|_{L^2}=\|f\|_{L^2} \quad \text{and} \quad \|U(t_0,t)f\|_{H^1}=\|f\|_{H^1}.\notag
\eeq
For more properties of $U(t_0,t)$, see Appendix \ref{sec:appendix}.

\medskip
Note that, given $u_0 \in H^1(\R)$, there exists a unique solution $v \in \calC((-T_-, T_+), H^1)$ for equation \eqref{eq:duhamel fomula} by Corollary \ref{cor:maximal existence in H^1}. If we let $u=T_{D(t)}v$, then $u\in \calC((-T_-, T_+), H^1)$ solves equation \eqref{eq:duhamel fomula for o} and the blow--up alternative holds since $\{T_{D(t)}: t\in \R \}$ is a strongly continuous group of unitary operators on $H^1(\R)$.
Moreover, since $H^1(\R)\hookrightarrow L^\infty(\R)$, we have $u\in L^\infty((-T_-, T_+), L^2) \cap L^4_\loc((-T_-, T_+), L^\infty)$ and, therefore, by
the Riesz--Thorin interpolation Theorem,
\beq\label{eq:mixed space}
u\in L^q_\loc((-T_-, T_+), L^p)
\eeq
for every admissible pair $(p,q)$. Before we prove the global existence of a solution, we show the existence of a unique global solution of \eqref{eq:duhamel fomula for o} with the initial datum $u_0\in L^2(\R)$ when $\dav\neq \pm 1$.
As usual, we  use the Strichartz estimates(Lemma \ref{lem:strichartz}) to prove the existence of a local solution for \eqref{eq:duhamel fomula for o}, see \cite{Cazenave, Kato1987}  for example.

\begin{proposition}\label{prop:L^2theory}
Let $\dav\neq \pm 1$.
 For any  $u_0\in L^2(\R)$, there exists a unique global solution $u\in \calC(\R, L^2)\cap L^6_{\loc}(\R, L^6)$  of \eqref{eq:duhamel fomula for o}. Moreover, the solution $u$  satisfies
 $$
\|u(t)\|_{L^2}=\|u_0\|_{L^2} \quad \text{for all }  t\in \R.
$$
Furthermore, for every $M>0$ and admissible pair  $(p,q)$, there exists a positive constant $C$ depending on $\dav$ and $\|u_0\|_{L^2}$ such that
\beq\label{est:solutionL^pL^q}
\|u\|_{L^q([-M,M], L^p)} \leq C.
\eeq
\end{proposition}

\begin{proof}
Let $0\neq u_0\in L^2(\R)$ be fixed. Without loss of generality, we consider positive $t$ only.
First, to prove the existence of a unique solution  in $\calC([0,1], L^2)\cap L^6([0,1], L^6)$ of \eqref{eq:duhamel fomula for o}, let us define the closed ball
\beq\label{def:B}
B_{M,a}:=\{u\in L^\infty([0,M], L^2) \cap L^6([0,M], L^6) \; : \; \|u\|_{L^\infty([0,M], L^2)} + \|u\|_{ L^6([0, M], L^6)}\leq a\} \notag
\eeq
equipped with the distance
$$
d(u,v)=\|u-v\|_{ L^\infty([0, M], L^2)}+\|u-v\|_{ L^6([0, M], L^6)}
$$
for each $0<M \leq 1$ and $a>0$.
Define the map $\Phi$ on $B_{M,a}$ by
\[
\Phi(u)(t)=U(0,t)u_0 + i \int _0^t U(t',t) \Bigl( G(t')|u(t')|^2u(t')\Bigr)dt'.
\]
For appropriate values of $M$ and $a$, the map  $\Phi$ is a contraction  on $(B_{M,a},d)$. Indeed, it follows from the Stricharz estimates(Lemma \ref{lem:strichartz}) and the Cauchy--Schwarz  inequality that
\beq
\begin{aligned}\label{ineq:L^2proofHolder}
\|\Phi(u)\|_{L^6([0, M], L^6)}&\leq  C\|u_0\|_{L^2}+ C\| G(\cdot) |u|^2u\|_{L^{1}([0, M], L^{2 })}\\
& \leq C \|u_0\|_{L^2}+ C M^{1/2} \| u\|_{L^6([0, M], L^6)}^3\ .
\end{aligned}
\eeq
On the other hand, using the unitarity of $U(0,t)$ on $L^2(\R)$ and the argument used in \eqref{ineq:L^2proofHolder},  we obtain
\[
\begin{aligned}
\|\Phi(u)\|_{L^\infty([0, M],L^2)}&\leq  \|u_0\|_{L^2}+  C\|G(\cdot) |u|^2u\|_{L^{1}([0, M], L^{2 })} \\
& \leq \|u_0\|_{L^2}+ C M^{1/2} \| u\|_{L^6([0, M], L^6)}^3.
\end{aligned}
\]
Next, noting
\beq\label{eq:difference of cubes}
||z_1|^2z_1-|z_2|^2z_2|\leq C (|z_1|^2+|z_2|^2)|z_1-z_2| \quad\mbox{for all }z_1, z_2\in \C,
\eeq
 by the same arguments above, we get
\begin{align}\label{ineq:distance}
\| \Phi(u) -\Phi(v) \| _{L^6([0, M],L^6)} &\leq C
\|G(\cdot)(|u|^2u-|v|^2v)\|_{L^1([0, M], L^{2  })} \notag\\
&\leq C \int_0^M (\|u(t)\|_{L^6}^2 + \|v(t)\|_{L^6}^2)\|u(t)-v(t)\|_{L^6} dt\\
& \leq CM^{1/2}(\|u\|_{L^6([0, M],L^6)}^2 + \|v\|_{L^6([0, M],L^6)}^2 ) \|u-v\|_{L^6([0, M],L^6)} \notag
\end{align}
and
\[
\begin{aligned}
\|\Phi(u)-\Phi(v)\|_{L^\infty([0, M],L^2)}&\leq C
\|G(\cdot)(|u|^2u-|v|^2v)\|_{L^1([0, M], L^{2  })} \\
& \leq CM^{1/2}\left(\|u\|_{L^6([0, M],L^6)}^2 + \|v\|_{L^6([0, M],L^6)}^2 \right) \|u-v\|_{L^6([0, M],L^6)}.
\end{aligned}
\]
Therefore, we have a positive constant $C$ such that  for all $u, v\in B_{M,a}$,
\[
\|\Phi(u)\|_{L^\infty([0, M],L^2)}+\|\Phi(u)\|_{L^6([0, M], L^6)}  \leq  C \|u_0\|_{L^2} + CM^{1/2}a^3
\]
and
\beq \label{ineq:contraction}
d(\Phi(u),\Phi(v))\leq CM^{1/2}a^2 d(u,v).\notag
\eeq
Now set $a=2C\|u_0\|_{L^2}$ and choose $0<M\leq 1$ satisfying
\beq \label{choice:M}
C M^{1/2}(2C\|u_0\|_{L^2})^2<  \f12, \notag
\eeq
then we obtain that $\Phi$ is a contraction from $B_{M,2C\|u_0\|_{L^2}}$ into itself. Thus, $\Phi$ has a unique fixed point $u\in B_{M,2C\|u_0\|_{L^2}}$ and
\beq \label{ineq:bound in S}
\|u\|_{L^\infty([0, M],L^2)}+ \|u\|_{L^6([0, M],L^6)}\leq 2C\|u_0\|_{L^2}.
\eeq
Moreover, the Stricharz estimates guarantee that $u$ is even in $\calC([0,M], L^2) \cap L^6([0,M], L^6) $.

To prove the uniqueness in $\calC([0,M], L^2) \cap L^6([0,M], L^6)$, it is enough to find a small $\delta>0$ so that the uniqueness is guaranteed in $\calC([0,\delta], L^2) \cap L^6([0,\delta], L^6)$. Suppose that
$u_1, u_2 \in \calC([0,\delta], L^2) \cap L^6([0,\delta], L^6) $ solve \eqref{eq:duhamel fomula for o}, for $0< \delta \le M$.
It follows from \eqref{ineq:distance} that
\[
\|u_1-u_2\|_{L^6([0,\delta], L^6)}
 \leq
C\delta^{1/2} \left(\|u_1\|^2_{L^6([0,\delta], L^6)}+\|u_2\|^2_{L^6([0,\delta], L^6)}\right)\|u_1-u_2\|_{L^6([0,\delta], L^6)}.
\]
Choosing $\delta>0$ small enough, we obtain
\[
\|u_1-u_2\|_{L^6([0,\delta], L^6)}\leq \f 1 2  \|u_1-u_2\|_{L^6([0,\delta], L^6)},
\]
which implies that $u_1=u_2$ on $\calC([0,\delta], L^2) \cap L^6([0,\delta], L^6)$.

Let $u\in \calC([0,M], L^2) \cap L^6([0,M], L^6)$ be a unique solution of \eqref{eq:duhamel fomula for o}. We now show that $u\in L^q([0,  M], L^p)$
for every admissible pair $(p,q)$.
Applying the Strichartz estimates to \eqref{eq:duhamel fomula for o} and the H\"older inequality with exponents
$\f{5}{2}$ and $\f{5}{3}$ in $x-$integral and $t-$integral,
we get
\[
\begin{aligned}
\|u\|_{L^q([0 ,  M], L^p)}&\leq  C\|u_0\|_{L^2}+ C\| G(\cdot) |u|^2u\|_{L^{6/5}([0 , M], L^{6/5})}\\
& \leq C \|u_0\|_{L^2}+ C\left(\int_0^{  M} \|u(t)\|_{L^2}^{6/5}\|u(t)\|_{L^6}^{12/5}dt\right)^{5/6}
\\
& \leq C \|u_0\|_{L^2}+ CM^{1/2} \|u\|_{L^\infty([0,M],L^2)} \| u\|_{L^6([0, M], L^6)}^2
\\
& \leq C \|u_0\|_{L^2}+ CM^{1/2}\|u_0\|^3_{L^2}
,
\end{aligned}
\]
where \eqref{ineq:bound in S} is used. Moreover, since $u$ satisfies the mass conservation, that is,
$$
\|u(t)\|_{L^2}=\|u_0\|_{L^2} \quad \text{for all } t \in [0,M],
$$
one can iterate this argument to obtain a unique solution in $\calC([0,1], L^2)\cap L^6([0,1], L^6)$ of \eqref{eq:duhamel fomula for o}.
Iterating this process, we obtain a unique global solution and it satisfies \eqref{est:solutionL^pL^q}.
\end{proof}

Now we give the proof of Theorem \ref{thm:globalwellposedness}.

\begin{proof}[Proof of Theorem \ref{thm:globalwellposedness}]
Let $u_0\in H^1(\R)$ be fixed and let us consider positive times only. \\
We first consider the case $\dav\neq \pm 1$.
Let $u\in \calC([0, \infty), L^2)\cap L^6_{\loc}([0, \infty), L^6)$ be the solution of \eqref{eq:duhamel fomula for o}, obtained in Proposition \ref{prop:L^2theory}.
On the other hand, Corollary \ref{cor:maximal existence in H^1} gives us the solution $\tilde{u}\in \calC( [0,T_+), H^1)$ for \eqref{eq:duhamel fomula for o}, where $T_+>0$ is the maximal life time.
Moreover, we have $\tilde{u}\in \calC( [0,T_+), L^2) \cap L^6_{\loc}([0,T_+), L^6)$ from \eqref{eq:mixed space}.
By the uniqueness of a solution in  $\calC( [0,T_+),L^2) \cap L^6_{\loc}([0,T_+), L^6)$, $u=\tilde{u}$ and, therefore, $u$ is also in $\calC( [0,T_+), H^1)$ and it remains to show $T_+=\infty$.

Suppose to the contrary that $T_+< \infty$ and
choose $n \in \N_0$ such that $n< T_+ \le n+1$.
It follows from \eqref{est:solutionL^pL^q} that
\beq\label{est:mixed norm}
\|u\|_{L^q([n,T_+], L^p)} < \infty
\eeq
for every admissible pair $(p,q)$.

For now, we assume to have
\beq\label{eq:formula}
\f{d}{dt} \| \partial_x u(t)\|_{L^2}^2 = - 2G(t)\im \int _\R \Bigl(u(t)\overline{\partial_x u(t)}\Bigr)^2 dx
\eeq
for all $t \in (n,T_+)$.
Then, for every $t \in (n,T_+)$,
$$
\f{d}{dt} \| \partial_x u(t)\|_{L^2}^2 \leq  2 \|u(t)\|_{L^\infty}^2 \| \partial_x u(t)\|_{L^2}^2
$$
and, therefore, by Gronwall's inequality and the Cauchy--Schwarz inequality, we have
\[
\begin{aligned}
\| \partial_x u(t)\|_{L^2}^2 &\leq \|\partial_x u(n)\|_{L^2}^2 \exp\left(2\int_n ^t  \|u(t')\|_{L^\infty}^2dt'\right)\\
&\leq  \|\partial_x u(n)\|_{L^2}^2\exp\left( 2\|u\|_{L^4([n, T_+), L^\infty)}^2 \right).
\end{aligned}
\]
Thus, we use \eqref{est:mixed norm} and the mass conservation to obtain
\beq\label{eq:bounded}
\sup_{n < t < T_+}\| u (t)\|_{H^1}^2  < \infty
\eeq
which contradicts the blow--up alternative.

To finish this case, it remains to show \eqref{eq:formula}. We use the twisted solution
\[
w(t)=(U(n,t))^{-1}u(t).
\]
Since $u$ solves \eqref{eq:duhamel fomula for o}, $w$ solves
\beq
w(t)= u(n) +i\int_n ^t (U(n,t'))^{-1}G(t')|u(t')|^2u(t')dt' \notag
\eeq
and, therefore, $w$ is differentiable on $(n, T_+)$ and
\beq
\dot{w}  (t) =\partial_t w(t)=i(U(n,t))^{-1}G(t)|u(t)|^2u(t)\notag
\eeq
which is in $H^1(\R)$.
Using this, one sees that
\[
\begin{aligned}
\f{d}{dt} \| \partial_x u(t)\|_{L^2}^2  = \f{d}{dt} \| \partial_x w(t)\|_{L^2}^2  &=2 \re  \la \partial_x w(t), \partial_x \dot{w}  (t) \ra \\
&= -2 G(t)\im \la \partial_x  u(t), \partial _x (|u(t)|^2 u(t))\ra ,
\end{aligned}
\]
which yields \eqref{eq:formula}.

Next, we assume that $\dav = 1$ and solve \eqref{eq:mainCauchy} recursively. First, we find a solution in $\calC([0,1], H^1)$ of \eqref{eq:mainCauchy}, i.e.,
\[
\begin{cases}
 i\partial_t u + 2\partial^2_x u+ G(t)|u|^2u =0 &\mbox{for }t\in (0, 1),\\
 u=u_0 & \mbox{for } t=0.
\end{cases}
\]
It is well--known that there exists a unique solution $u \in \calC([0,1], H^1)$, for example, see \cite{CS}. Denote $u(1)$ by $u_1$ and solve the ordinary differential equation with the initial datum $u_1$
\beq\label{eq:ode}
\begin{cases}
  i\partial_t u + G(t)|u|^2u =0 & \mbox{for }t\in (1, 2) \\
  u=u_1  & \mbox{for } t=1
\end{cases}
\eeq
of which solution is given by
\beq\label{eq:solution for ode}
u(t)=u_1\exp {\left(i|u_1|\int_1 ^tG(t')dt'\right)}.
\eeq
Here, we used that $|u(t)|=|u_1|$ for all $t \in (1,2)$. Indeed, multiplying the ordinary equation in \eqref{eq:ode}  by $\overline{u}$ and taking the imaginary part of the resulting identity, we obtain $\partial_t |u(t)|^2=0$ for all $t \in (1,2)$. Noting that \eqref{eq:solution for ode} has the limit at $t=2$, we continuously extend $u$ to $[1,2]$. Repeating this process completes the proof. The case $\dav =-1$ is done similarly.
\end{proof}

\begin{remark}
To get \eqref{eq:formula}, one can use a formal calculation as follows:
\beq
\f{d}{dt} \| \partial_x u(t)\|_{L^2}^2 =-2\re \la \partial_x u, \partial_x \partial_t u \ra =
2\im \left((\dav+d_0(t))\la \partial_x u, \partial_x\partial_x^2 u\ra
+  G(t) \la \partial_x u, \partial_x |u|^2u\ra\right). \notag
\eeq
However, the scalar product $\la \partial_x u, \partial_x\partial_x^2 u\ra$  in $L^2(\R)$ may not be defined for $u\in H^1(\R)$  since $\partial_x u\in L^2(\R)$ and $\partial_x\partial_x^2 u=\partial_x^3 u\in H^{-2}(\R)$.
 Thus,
as in \cite{CHLW}, we used the twisting argument, see, also, \cite{ AHH, O06}.
\end{remark}

{
\section{Averaging theorem}\label{sec:Averaging theorem}

In this section, to prove the averaging theorem (Theorem \ref{thm:intro averaging}),
we compare the solutions of equations \eqref{eq:intro transformed equation} and \eqref{eq:intro averaged}
\beq \label{eq:transformed eq}
i\partial_t v + \dav \partial_{x}^2 v  + Q_\veps (v)=0,
\eeq
and
\beq \label{eq:averaged eq}
i\partial_t v + \dav \partial_{x}^2 v + \la Q\ra (v)=0
\eeq
with close initial data $u_0, v_0 \in H^1(\R)$, where the nonlinearities are given by
$$
 Q_\veps (v):=G\left(\f t\veps\right) T_{D(t/\veps)}^{-1} \Bigl(|T_{D(t/\veps)}v|^2 T_{D(t/\veps)}v\Bigr)
$$
and
$$
\la Q\ra (v) := \int_0^1  T_r^{-1} \Bigl(|T_r v|^2 T_r v\Bigr)\psi(r)dr,
$$
respectively.
Recall that
$$
\psi(r)=e^{-\Gamma}\cosh \Gamma(r-1).
$$
We, first, show a lemma which is inspired by the proof of Theorem 4.1 in \cite{ZGJT01}.
\begin{lemma}\label{lemma:averaging}
For every $M>0$, if $v\in \calC([-M,M], H^1)$, then
\beq\label{conv:averaging}
\int_0^t e^{i \dav (t-t')\partial_x^2 } Q_\veps (v(t')) dt'  \to
\int_0^t e^{i \dav (t-t')\partial_x^2 } \la Q\ra (v(t')) dt'
\eeq
in $\calC([-M, M], H^{1})$ as $\veps \to 0$.
\end{lemma}

\begin{proof}

We consider positive times only. Let $R_\veps(v):=  Q_\veps (v )-\la Q\ra (v)$. Then, by the Plancherel identity, its Fourier transform in $x$ can be expressed by
\beq
\widehat{R_\veps(v)}( \xi,t) = \f{1}{2\pi}\int_{\R^3} \delta(\xi_1-\xi_2 +\xi_3 -\xi)
A_\veps (\xi_1,\xi_2 ,\xi_3 ,\xi, t )  \hat{v}(\xi_1,t) \overline{ \hat{v}( \xi_2,t) } \hat{v}(\xi_3,t)  d\xi_1 d\xi_2d\xi_3 \notag
\eeq
for all $\xi \in \R$ and $t \in [0,M]$, where
$$
    A_\veps (\xi_1,\xi_2 ,\xi_3 ,\xi, t ) := G(t/\veps) e^{-i D(t/\veps)(\xi_1^2-\xi_2^2 +\xi_3^2 -\xi^2)}-\int _0^1 e^{-ir(\xi_1^2-\xi_2^2 +\xi_3^2-\xi^2)}\psi(r) dr .
$$
Now, we define $B_\veps:\R^4\times[0,M] \to \C$ by
\beq
B_\veps (\xi_1,\xi_2 ,\xi_3 ,\xi, t ) : =\int _0 ^t A_\veps (\xi_1,\xi_2 ,\xi_3 ,\xi, t' ) dt' , \notag
\eeq
then
\[
  B_\veps (\xi_1,\xi_2 ,\xi_3 ,\xi, t )  =\int _0 ^t A_1 (\xi_1,\xi_2 ,\xi_3 ,\xi, t'/\veps ) dt'
   =\veps \int _0 ^{t/\veps} A_1 (\xi_1,\xi_2 ,\xi_3 ,\xi, t' ) dt',
\]
and therefore,
\[
|B_\veps (\xi_1,\xi_2 ,\xi_3 ,\xi, t )|
\le \veps \int _0 ^2 |A_1 (\xi_1,\xi_2 ,\xi_3 ,\xi, t')| dt'
\le 4\veps
\]
since $A_1$ is a $2-$periodic function in $t$ with mean zero and bounded by $2$. Thus,
\beq\label{B}
\|B_\veps  \|_{L^\infty (\R^4\times[0,M] )} \leq 4\veps.
\eeq
 Using the same argument as in \eqref{ineq:Lips},  for every $v_1, v_2\in \calC([0,M], H^1(\R))$, we see
\beq \label{ineq:density}
\begin{aligned}
&\left\|\int_0 ^\cdot e^{i\dav (\cdot-t')\partial_x^2} \left(Q_\veps (v_1)(t')- Q_\veps (v_2)(t')\right)dt' \right\|_{L^\infty([0, M], H^1)}
\leq CM  \|v_1-v_2\|_{L^\infty([0,M], H^1)}.
\end{aligned}
\eeq
The estimate \eqref{ineq:density} also holds when $Q_\veps$ is replaced by $\la Q\ra$.
Therefore, by a density argument, it suffices to prove \eqref{conv:averaging} for $v\in \calC^1([0, M], \mathcal{S}(\R))$ only.
Since
\[
\f{\partial}{\partial t}B_\veps(\xi_1,\xi_2 ,\xi_3 ,\xi,t)= A_\veps(\xi_1,\xi_2 ,\xi_3 ,\xi ,t)
\]
for almost every $t\in [0, M]$, by the integration by parts, we obtain
\[
\int_0^t  e^{i\dav (t-t')\xi^2} \widehat{R_\veps(v)}( \xi,t') dt'
=\widehat{ \calI_1(v)} (\xi,t)- \int_0^t  e^{i\dav (t-t')\xi^2} \left(\widehat{\calI_2(v)} (\xi,t')-  i \dav \widehat{\calI_3(v)} (\xi,t')\right)dt',
\]
where
\[
\widehat{ \calI_1(v)}(\xi,t)= \int_{\R^3}  \delta(\xi_1-\xi_2 +\xi_3 -\xi)B_\veps(\xi_1,\xi_2 ,\xi_3 ,\xi,t)\hat{v}(\xi_1,t) \overline{ \hat{v}( \xi_2,t) } \hat{v}(\xi_3,t)  d\xi_1 d\xi_2d\xi_3,
\]
\[
\widehat{\calI_2(v)} (\xi,t)=  \int_{\R^3}  \delta(\xi_1-\xi_2 +\xi_3 -\xi) B_\veps(\xi_1,\xi_2 ,\xi_3 ,\xi,t)
\partial_{t} \bigl(\hat{v}(\xi_1,t) \overline{ \hat{v}( \xi_2,t) } \hat{v}(\xi_3,t)\bigr)
d\xi_1 d\xi_2d\xi_3 ,
\]
  and
\[
\widehat{\calI_3(v)}(\xi,t)=\int_{\R^3}  \delta(\xi_1-\xi_2 +\xi_3 -\xi)
B_\veps(\xi_1,\xi_2 ,\xi_3 ,\xi,t)\xi^2
\hat{v}(\xi_1,t) \overline{ \hat{v}( \xi_2,t) } \hat{v}(\xi_3,t)
d\xi_1 d\xi_2d\xi_3.
\]
Then, fix $t\in [0, M]$, we have
\begin{align}\label{eq:limit}
& \left\| \int_0^t e^{i \dav (t-t')\partial_x^2 } R_\veps(v)(\cdot, t') dt'\right\|_{H^{1}}
= \left(\int_\R(1+\xi^2)\left|\int_0^t e^{i \dav (t-t')\xi^2 } \widehat{R_\veps(v)}(\xi, t') dt'\right|^2 d\xi\right)^{1/2} \notag\\
& \lesssim \|\calI_1(v)(\cdot,t)\|_{H^1} +
 \int_0 ^t  \left(\|\calI_2(v)(\cdot,t')\|_{H^1} + \|\calI_3(v)(\cdot,t')\|_{H^1}\right)dt',
  \end{align}
where we use Minkowski's inequality.
First, to get a bound of $\calI_1(v)$, note
\beq
\begin{aligned}
|\widehat{ \calI_1(v)}(\xi,t)|\leq
\|B_\veps\|_{L^\infty (\R^4\times[0,M] )} \left|\int_{\R^2}  \hat{v}(\xi_1,t) \overline{ \hat{v}( \xi_2,t) } \hat{v}(\xi-\xi_1+\xi_2,t)  d\xi_1 d\xi_2\right|\notag
\end{aligned}
\eeq
for all $\xi$ and $t$.
If we define $I(f_1, f_2, f_3)$ by its Fourier transform
\[
\widehat{I}(f_1, f_2, f_3)(\xi)= \int_{\R^2} \hat{f_1}(\xi_1) \overline{ \hat{f_2}( \xi_2) } \hat{f_3}(\xi-\xi_1+\xi_2)  d\xi_1 d\xi_2
\]
for all $f_1, f_2, f_3 \in \calS(\R)$, then by a straightforward calculation, we obtain
\beq \label{ineq:I}
\|I(f_1, f_2, f_3)\|_{H^1}\leq C\|f_1\|_{H^1} \|f_2\|_{H^1} \|f_3\|_{H^1}.\notag
\eeq
This together with \eqref{B}, we have
\beq\label{ineq:I1}
\|\calI_1(v)(\cdot, t)\|_{H^1} \le \|B_\veps\|_{L^\infty (\R^4\times[0,M] )}\|I(v(t),v(t),v(t))\|_{H^1}
\lesssim  \veps \|v(t)\|_{H^1}^3\notag
\eeq
for all $t$. By a similar argument, we have
\[
\|\calI_2(v)(\cdot, t)\|_{H^1} \lesssim  \veps \|\partial_t v(t)\|_{H^1} \|v(t)\|_{H^1}^2
\]
and
\[
\|\calI_3(v)(\cdot, t)\|_{H^1}
\lesssim  \veps \Bigl(\|\partial_x^2 v(t)\|_{H^1} \|v(t)\|_{H^1}^2
+ \|\partial_x v(t)\|_{H^1}^2 \|v(t)\|_{H^1}\Bigr)
\]
for all $t$. Substituting the last three inequalities into \eqref{eq:limit} completes the proof.
\end{proof}

 Now we are ready to give
\begin{proof}[Proof of Theorem \ref{thm:intro averaging}]
Fix $M>0$ and consider positive times only.
Let
\beq
K=2\sup_{t\in [0, M]}\|v(t)\|_{H^1} \notag
\eeq
and let $0< \veps \leq \f{K}{2}$ for now. Then we have $\|u_0\|_{H^1}\leq K$
since  $\|u_0-v_0\|_{H^1}\leq \veps$ and $\|v_0\|_{H^1}\leq \f{K}{2}$.
Then, it follows from Proposition \ref{prop:H^slocal theory} that there exists $M_+= M_+(K)$, independent of $\veps$, such that
$v_\veps\in \calC([0, M_+], H^1(\R))$ and
\beq\label{sup:veps}
\sup_{0<\veps\leq \f K 2}\sup_{t \in [0, M_+]} \|v_\veps(t) \|_{H^{1}} \leq 2K.
\eeq
We now prove that \eqref{ineq:averaging thm} holds on $[0, M_+]$, i.e., there exists $C>0$ such that
\[
\|v_\veps - v\|_{\calC([0, M+],H^{1}(\R))}\leq C \veps.
\]
By Duhamel's formula, we have
\[
v_\veps(t)-v(t) = e^{i\dav t \partial^2_x}(u_0-v_0)+i \mathcal{I}_1(t)+i\mathcal{I}_2(t)
\]
for all $0\leq t \leq M_+$, where
\[
\mathcal{I}_1(t)=\int_0^t  e^{i\dav (t-t')\partial_x^2}\left[Q_\veps (v_\veps  )(t')
-Q_\veps (v)(t')\right] dt'
\]
and
\[
\mathcal{I}_2(t)=\int_0^t  e^{i\dav (t-t')\partial_x^2}\left[Q_\veps (v)(t')-
 \la Q\ra (v)(t') \right] dt'.
\]
It follows from Lemma \ref{lemma:averaging} that there exists a constant $C>0$ such that
\beq\label{eq:I1}
\begin{aligned}
\sup_{t\in [0, M_+]}\left\|\mathcal{I}_2(t) \right\|_{H^{1}}  \leq C\veps.
\end{aligned}
\eeq
To bound $\mathcal{I}_1$, we use Minkowski's inequality and Lemma \ref{lem:Q}, then we obtain
\beq\label{eq:I2}
\begin{aligned}
\left\|\mathcal{I}_1(t) \right\|_{H^{1}}  & \leq  \int_0^t \|Q_\veps (v_\veps)(t') -Q_\veps (v )(t') \|_{H^{1}} dt' \\
& \lesssim  \int_0^t \left(\|v_\veps(t') \|^2_{H^{1}}+\|v(t')\|^2_{H^{1}} \right)\|v_\veps(t') -v(t') \|_{H^{1}} dt'\\
\end{aligned}
\eeq
for all $0\leq t \leq M_+$.
Since  $\|u_0-v_0\|_{H^1}\leq \veps$, it follows from \eqref{eq:I1} and  \eqref{eq:I2} that, for all $0\leq t \leq M_+$, there exist positive constants $C_1$, depending only on $K$, and $C_2$  such that
\begin{align*}
  \|v_\veps(t)-v(t)\|_{H^{1}} & \le \|u_0-v_0\|_{H^1}+ \left\|\mathcal{I}_1(t) \right\|_{H^{1}} +\left\|\mathcal{I}_2(t) \right\|_{H^{1}}\\
   & \le C_2\veps + C_1 \int_0^t \|v_\veps(t') -v(t') \|_{H^{1}} dt'.
\end{align*}
Thus, by Gronwall's inequality, we obtain
\beq\label{conv:M_+}
\sup_{t \in [0, M_+]}\|v_\veps(t)-v(t)\|_{H^{1}} \leq C_2\veps e^{C_1M_+}.
\eeq
If $M_+\geq M$, the proof is complete and now we assume that $M_+<M$. It follows from \eqref{conv:M_+} and $v_\veps-v \in \calC([0, M_+], H^1(\R))$ that
\[
 \|v_\veps(M_+)-v(M_+)\|_{H^{1}}\leq  C_2\veps e^{C_1M_+}.
\]
Now choose $\veps_1>0$ such that $C_2\veps_1 e^{C_1M_+}\leq \f{K}{2}$ and $\veps_1\leq \f{K}{2}$. Let $0< \veps \leq \veps_1$. Then
\[
 \|v_\veps(M_+)\|_{H^{1}}\leq  \f{K}{2}+C_2\veps e^{C_1M_+}\leq K.
\]
Applying Proposition \ref{prop:H^slocal theory} to \eqref{eq:duhamel fomula} with the initial datum $v_\veps(M_+)$, which satisfies
\[
\sup_{0< \veps \leq \veps_1} \|v_\veps(M_+)\|_{H^1}\leq K,
\]  we have
\[
\sup_{0<\veps\leq \veps_1} \sup_{t\in [M_+, 2M_+]}\|v_\veps(t)\|_{H^1}\leq 2K.
\]
Combining this and \eqref{sup:veps}, we have
\beq\label{bound:veps_1}
\sup_{0<\veps\leq \veps_1}\sup_{t \in [0, 2M_+]} \|v_\veps(t) \|_{H^{1}} \leq 2K. \notag
\eeq
Repeat the above argument to obtain
\beq
\sup_{t \in [0, 2M_+]}\|v_\veps(t)-v(t)\|_{H^{1}} \leq C_2\veps e^{2C_1M_+}. \notag
\eeq
Iterating this argument, 
we can choose $\veps_0>0$ such that $C_2\veps_0 \exp\left((\lfloor \f {M} {M_+}\rfloor +1)C_1M_+ \right)\leq \f{K}{2}$ and $\veps_0\leq \f{K}{2}$ to get \eqref{ineq:averaging thm} with $C= C_2 \exp\left((\lfloor\f {M} {M_+} \rfloor+2)C_1M_+\right)$, where $\lfloor a \rfloor=\max \{n\in \Z : n\leq a\}$.
This completes the proof.
\end{proof}



%
\appendix
\setcounter{section}{0}
\renewcommand{\thesection}{\Alph{section}}
\renewcommand{\theequation}{\thesection.\arabic{equation}}
\renewcommand{\thetheorem}{\thesection.\arabic{theorem}}

\section{Linear propagator}\label{sec:appendix}
\setcounter{theorem}{0}
\setcounter{equation}{0}

For the reader's convenience, the properties of the linear propagator $U(t_0,t) $ defined in  \eqref{eq:linear-propagator}  are referred in this section, which can be also found in \cite{AK, ASS}.

\begin{lemma} \label{lem:dispersivelemma}
Let $\dav \neq \pm 1$. Then there exists a constant $C$ depending only on $\dav$ such that
\beq \label{ineq:dispersive}
\|U(t_0,t)f\|_{L^\infty}\leq C|t-t_0|^{-1/2}\|f\|_{L^1}\notag
\eeq
for all $f\in L^1(\R)$ and all distinct $t,t_0 $ with $\lfloor t \rfloor = \lfloor t_0\rfloor$.

\end{lemma}
\begin{proof}
Using the kernel of the
solution operator for the free Schr\"odinger equation
\[
(T_tf)(x)=\f{1}{\sqrt{4\pi i t}}\int_\R e^{i\f{|x-y|^2}{4t}}f(y)dy, \quad t\neq 0,
\]
for Schwartz functions $f$,
we obtain that
\beq \label{ineq:dispersiveproof}
\|U(t_0,t)f \|_{L^\infty}\leq  \f{\|f\|_{L^1}}{(4\pi|\int_{t_0}^t d(t')dt'|)^{1/2}}, \quad t\neq t_0. \notag
\eeq
Using this and the fact that for all $t$ and $t_0$ with $\lfloor t \rfloor = \lfloor t_0 \rfloor$,
$$
\left| \int_{t_0}^t d(t')dt'\right|= \left\{
            \begin{array}{ll}
              |\dav+1||t-t_0|, & \lfloor t_0 \rfloor \mbox{ even}, \\
\\
              |\dav-1||t-t_0|, &  \lfloor t_0 \rfloor \mbox{ odd},
            \end{array}
          \right.
$$
we complete the proof.
\end{proof}
Using Lemma \ref{lem:dispersivelemma} and the unitarity of $U(t_0,t)$ on $L^2(\R)$, via the well--known arguments, we obtain the one--dimensional Strichartz estimates on each  time interval $[n, n+1]$, for every $n\in \Z$. For the classical Strichartz estimates, see, e.g., \cite{Strichartz, GV, KT}. As usual, we say that a pair of exponents $(p,q)$  is admissible  if $2 \le p \le \infty$ and
\beq\label{admissible pair}
\frac{1}{p}+\frac{2}{q}=\frac{1}{2}.\notag
\eeq
and, also, for every $p\geq 1 $, denote by $p'$ the H\"older conjugate, i.e.,
$$
\f{1}{p}+\f{1}{p '}=1.
$$
\begin{lemma} \label{lem:strichartz}
Assume $\dav\neq \pm 1$. Let $(p,q)$ and $(p_0, q_0)$ be admissible pairs and $t_0\in [n,n+1]$ for some $n\in \Z$.
\begin{theoremlist}
\item  If $f\in L^2(\R)$,  then the function $t \mapsto U(t_0,t) f$ on $[n, n+1]$ belongs to $L^q([n, n+1], L^p)\cap \calC([n, n+1], L^2)$. Moreover, there exists a constant $C$ depending only on $p$ and $\dav$ such that for all $f\in L^2(\R)$
\beq \label{est:strichartz est 1}
\|U(t_0,t) f\|_{L^q([n, n+1], L^p)}\le C \|f\|_{L^2}.\notag
\eeq
\item
Let $I$ be an interval contained in $[n,n+1]$ and $t_0 \in  \ol{I}$. If $F\in L^{q'_0}(I,L^{p'_0})$, then the function
$$
t \mapsto\int_{t_0} ^t U(t',t)F(t')dt'
$$
on $I$ belongs to $L^q(I, L^p)\cap \calC(\ol{I}, L^2)$.  Moreover, there exists a constant $C$ depending only on $p, p_0$, and $\dav$ such that for all $F\in L^{q'_0}(I,L^{p'_0})$
\beq  \label{est:strichartz est 2}
\left(\int_I\left\|\int_{t_0} ^t U(t',t)F(\cdot, t')dt'\right\|_{L^{p}}^{q} dt \right)^{1/q}
\le C \left(\int_I \|F(\cdot, t)\|_{L^{p'_0}}^{q'_0}dt\right)^{1/q_0'}.\notag
\eeq
\end{theoremlist}
\end{lemma}


\noindent
\textbf{Acknowledgements: }

Young--Ran Lee and Mi--Ran Choi are supported by the National Research Foundation of Korea(NRF) grants funded by the Korean government(MSIT--2020R1A2C1A01010735 and MOE--2019R1I1A1A01058151).

\renewcommand{\thesection}{\arabic{chapter}.\arabic{section}}
\renewcommand{\theequation}{\arabic{chapter}.\arabic{section}.\arabic{equation}}
\renewcommand{\thetheorem}{\arabic{chapter}.\arabic{section}.\arabic{theorem}}

\def\cprime{$'$}



\begin{thebibliography}{100}
\small{


\bibitem{AB98} M.\ J.\ Ablowitz and G.\ Biondini,
    \textit{Multiscale pulse dynamics in communication systems
    with strong dispersion management.}
    Opt.\ Lett.\ \textbf{23} (1998), 1668--1670. \hfill

\bibitem{Ag} G.~P.~Agrawal, \textit{Nonlinear Fiber Optics} (Fifth Edition). Academic\ Press\ , San Diego, (2012).
    \hfill


\bibitem{AK} J.~Albert and E.~Kahlil, \textit{On the well-posedness of the Cauchy problem for some nonlocal nonlinear Schr\"{o}dinger equations.}
	Nonlinearity\ \textbf{30} (2017), 2308--2333. \hfill

\bibitem{AHH}
	I.~Anapolitanos, M.~Hott, and D.~Hundertmark,
	\textit{Derivation of the Hartree equation for compound Bose gases
	in the mean field limit.}
	Rev.\ Math.\ Phys.\ \textbf{29} (2017), no.~7, 1750022 (28 pages). 
    \hfill

\bibitem{ASS} P. ~ Antonelli, J.~ C. ~Saut, and C. ~Sparber, \textit{Well-posedness and averaging of NLS with time-periodic dispersion management.} Adv.\ Differential\ Equations\ \textbf{18} (2013), 49--68.\hfill

\bibitem{BD} A. ~ de Bouard and A. ~ Debussche, \textit{The nonlinear  Schr\"{o}dinger equation with white noise dispersion.}  J. Funct. Anal.\ \textbf{259}  (2010), no.~5, 1300--1321.\hfill

\bibitem{Cazenave} T.~Cazenave, \textit{Semilinear Schr\"{o}dinger Equations.} Courant Lecture Notes in Mathematics, vol. 10.
American\ Mathematical\ Society,\ Providence\ (2003).\hfill


\bibitem{CL} T.~Cazenave and P.~L.~Lions \textit{Orbital stability if standing waves for some nonlinear Schr\"{o}dinger Equations.}  Comm.\ Math.\ Phys.\ \textbf{85} (1982), 549--561.\hfill
	\hfill

\bibitem{CS} T. ~Cazenave and M. ~Scialom, \textit{A Schr\"{o}dinger equation with time-oscillating nonlinearity.} Rev. Mat. Complut.\ \textbf{23} (2010),  no.\  2,  321--339.\hfill

\bibitem{CHLT} M.--R.~Choi, D. Hundertmark, and Y.--R.~Lee, \textit{Thresholds for existence of dispersion management solitons for general nonlinearities}. SIAM\ J.\ Math.\ Anal.\ \textbf{49} (2017), no.~2, 1519--1569.
    \hfill

\bibitem{CHLW} M.--R.~Choi, D. Hundertmark, and Y.--R.~Lee, \textit{Well-posedness of dispersion managed nonlinear Schr\"{o}dinger equations}. arXiv:2003.09076 \hfill



\bibitem{CGTD93}  A.~R.~Chraplyvy, A.~H.~Gnauck,  R.~W.~Tkach, and  R.~M.~Derosier,
    \textit{$8\times10$ Gb/s transmission through 280 km of dispersion-managed fiber.}
    IEEE\ Photon.\ Technol.\ Lett.\ \textbf{5} (1993), 1233--1235. \hfill

\bibitem{EHL}
	M.~B.~Erdo\v{g}an, D.~Hundertmark, Y.--R.~Lee,
	\textit{Exponential decay of dispersion managed solitons
for vanishing average dispersion.}
	Math.\ Res.\ Lett.\ \textbf{18}(1) (2011), 13--26.\hfill

\bibitem{GT96a} I.~Gabitov and S.~K.~Turitsyn,
    \textit{Averaged pulse dynamics in a cascaded transmission system
    with passive dispersion compensation.}
    Opt.\ Lett.\  \textbf{21} (1996), 327--329.
    \hfill

\bibitem{GT96b} I.~Gabitov and S.~K.~Turitsyn,
    \textit{Breathing solitons in optical fiber links.}
    JETP Lett.\ \textbf{63} (1996), 861--866.
    \hfill

\bibitem{GV} J.~Ginibre and G.~Velo,
	\textit{The global Cauchy problem for the nonlinear Schr\"odinger equation revisited}.
	Ann. Inst. H. Poincar\'e Anal. Non Lin\'eaire \textbf{2} (1985), 309--327.
	\hfill

\bibitem{GH} W.~Green and D.~Hundertmark, \textit{Exponential decay of dispersion managed solitons for general dispersion profiles}.
	Lett.\ Math.\ Phys.\ \textbf{106} (2016), no.~2, 221--249.\hfill


\bibitem{HKS} D.~Hundertmark, P. ~Kunstmann, and R. ~Schnaubelt, \textit{Stability of dispersion managed solitons for vanishing average dispersion}.
	Arch.\ Math.\ \textbf{104} (2015), no.~3, 283--288.
	\hfill



\bibitem{HL} D. ~Hundertmark and Y.--R. ~Lee, \textit{ On non-local variational problems with lack of compactness related to non-linear optics}. J. Nonlinear Sci. \ \textbf{22} (2012), 1--38. \hfill

\bibitem{Kato1987}
T.~Kato, \textit{On nonlinear Schr\"odinger  equatuions.} Ann. Inst. H. Poincar\'{e} Phys. Th\'{e}or.\ \textbf{46}  (1987), 113--129.\hfill

\bibitem{KT} M. ~Keel  and T. ~Tao, \textit{ Endpoint Strichartz estimates.} Am. J. Math.\ \textbf{120} (1998), 955--980.\hfill

\bibitem{KH97} S.\ Kumar and A.\ Hasegawa,
    \textit{Quasi-soliton propagation in dispersion-managed optical fibers.}
     Opt.\ Lett.\ \textbf{22} (1997), 372--374.\hfill




\bibitem{LKC80} C.~Lin, H.~Kogelnik, and L.~G.~Cohen,
    \textit{Optical pulse equalization and low dispersion transmission in singlemode
            fibers in the 1.3--1.7 $\mu$m spectral region.}
    Opt.\ Lett.\ \textbf{5} (1980), 476--478.
    \hfill

\bibitem{O06} T.~Ozawa,
    \textit{Remarks on proofs of conservation laws for nonlinear Schr\"{o}dinger equations}
    Calc.\ Var.\ Partial Differential Equations\ \textbf{25} (2006), 403--408.
    \hfill










\bibitem{Strichartz}
    R.~S.~Strichartz,
    \textit{Restrictions of Fourier transforms to quadratic surfaces and
        decay of solutions of wave equations.}
    Duke Math.\ J.\ \textbf{44} (1977), 705--714.\hfill







\bibitem{TBF} S.~K.~Turitsyn, B.~Bale, and M.~P.~Fedoruk,
	\textit{Dispersion-managed solitons in fibre systems and lasers},
	Phys.\ Rep.\ \textbf{521} (2012), no.~4, 135--203.
	\hfill


\bibitem{T03} S.~K.~Turitsyn, E.~G.~Shapiro, S.~B.~Medvedev, M.~P.~Fedoruk, and V.~K.~Mezentsev,  \textit{Physics and mathematics of dispersion-managed optical solitons.}  C.\ R.\ Phys.\ \textbf{4} (2003), 145--161.\hfill

\bibitem{ZGJT01} V.~Zharnitsky, E.~Grenier, C.~K.~R.~T.~Jones, and S.~K.~Turitsyn,
\textit{Stabilizing effects of dispersion management.}
    Physica D\ \textbf{152-153} (2001), 794--817.\hfill


}
\end{thebibliography}
\end{document}